\documentclass[twoside, 12pt]{article}
\usepackage{amssymb, amsmath, mathrsfs, amsthm}
\usepackage{graphicx}
\usepackage{color}
\usepackage[top=2cm, bottom=2cm, left=2cm, right=2cm]{geometry}
\usepackage{float}

\newcommand{\bd}{\begin{description}}
\newcommand{\ed}{\end{description}}
\newcommand{\bi}{\begin{itemize}}
\newcommand{\ei}{\end{itemize}}
\newcommand{\be}{\begin{enumerate}}
\newcommand{\ee}{\end{enumerate}}
\newcommand{\beq}{\begin{equation}}
\newcommand{\eeq}{\end{equation}}
\newcommand{\beqs}{\begin{eqnarray*}}
\newcommand{\eeqs}{\end{eqnarray*}}

\definecolor{DarkGreen}{rgb}{0.2, 0.6, 0.3}

\newtheorem{theorem}{Theorem}

\newtheorem{observation}{Observation}
\newtheorem{lemma}{Lemma}

\newtheorem{corollary}{Corollary}
\newtheorem{case}{Case}

\def\endofClaim{\hfill\scalebox{.6}{$\blacksquare$}}
\newcommand{\oldqed}{}

\begin{document}
\title{On the structural growth of bipartite Ramsey numbers}

\author{
Meng Ji\footnote{School of Mathematical Sciences, and Institute of Mathematics and Interdisciplinary Sciences, Tianjin Normal University, Tianjin, China. {\tt
mji@tjnu.edu.cn}}
}
\date{}
\maketitle

\begin{abstract}
Bipartite Ramsey numbers is the smallest size of a complete bipartite graph $K_{N,N}$ such that every edge-coloring with a given number of colors inevitably yields a monochromatic copy of a prescribed bipartite graph. While exact values have been determined for certain specific graphs, the general asymptotic behavior of these numbers in terms of structural graph parameters remains poorly understood. In this paper, we investigate structure-dependent growth phenomena in bipartite Ramsey theory.

For a fixed bipartite graph $G$ with $p$ vertices and $q$ edges, we first establish  a lower bound of the form $\operatorname{br}(G,K_{n,n}) > C \bigl(\frac{n}{\log n}\bigr)^{(q-1)/(p-2)}$. As a corollary, we show that sufficiently dense bipartite graphs fail to be bipartite Ramsey size linear.
Turning to even cycles and complete bipartite graphs, we obtain an upper bound on the multicolor bipartite Ramsey number $\operatorname{br}_k(C_{2t};K_{n,n}) \le c_{t,k}\, n^2/\log^2 n$, which follows from classical estimates for Zarankiewicz numbers together with a double-counting argument. Building on this result, we further derive a refined linear upper bound of the form $\operatorname{br}(C_{2t},G) \le \frac{m}{2} + \frac{29t\sqrt{m}}{2}$, valid for any connected bipartite graph $G$ with $m$ edges and no isolated vertices.


\end{abstract}

\section{Introduction}

Ramsey theory investigates the inevitable emergence of monochromatic substructures in edge-colored graphs. 
In the bipartite setting, the bipartite Ramsey number $\operatorname{br}(G_1,G_2)$ is the smallest integer $N$ such that every red--blue coloring of the complete bipartite graph $K_{N,N}$ contains either a red copy of $G_1$ or a blue copy of $G_2$. 
More generally, for $q \ge 2$ and bipartite graphs $G_{1},\ldots,G_{q}$ on $n_{1},\ldots,n_{q}$ vertices, respectively, the $q$-color bipartite Ramsey number $\operatorname{br}(G_{1},\ldots,G_{q})$ is the minimum $N$ for which every $q$-coloring of the edges of $K_{N,N}$ yields a monochromatic subgraph isomorphic to $G_{i}$ in color $i$.

The systematic study of bipartite Ramsey numbers was initiated over fifty years ago by Faudree and Schelp \cite{Faudree-Schelp} and, independently, by Gy\'{a}rf\'{a}s and Lehel \cite{Gyarfas-Lehel}, who determined the exact value of $\operatorname{br}(P_{n},P_{n})$. 
Subsequent work has focused on cycles versus cycles; see, for example, \cite{Liu-Li,Yan-Peng,Zhang-Sun}. 
In 1975, Beineke and Schwenk \cite{Beineke-Schwenk} initiated the investigation of $\operatorname{br}(K_{s,s},K_{s,s})$. 
Currently, the best known bounds are the lower bound $\operatorname{br}(K_{s,s},K_{s,s})=\Omega\bigl((\sqrt{2})^{s}\bigr)$ from \cite{Hattingh-Henning} and the upper bound $\operatorname{br}(K_{s,s},K_{s,s})\leq O\bigl(2^{s+1}\bigr)$ established in \cite{Conlon-bipartite,Irving}. 
Results in the multicolor setting appear in \cite{Bucic-Letzter-Sudakov1,Bucic-Letzter-Sudakov2, Decamillis-Song}. 

Despite these advances, the general asymptotic behavior of $\operatorname{br}(G_{1},G_{2})$ in terms of structural parameters of $G_1$ and $G_2$ remains poorly understood. 
A central theme in this area is the study of \emph{size growth}: given a fixed graph $G$, how does $\operatorname{br}(G,H)$ grow as a function of the size of $H$? 
Erd\H{o}s, Faudree, Rousseau, and Schelp \cite{Erdos-Faudree-Rousseau-Schelp} systematically investigated size-linear growth of Ramsey numbers. 
A graph $H$ is called \emph{bipartite Ramsey size linear} if there exists a constant $C$ such that for every graph $G$ with $m$ edges and no isolated vertices, the bipartite Ramsey number $\operatorname{br}(H, G)$ is bounded above by $C\cdot m$. 
Identifying which bipartite graphs exhibit size-linear behavior is a fundamental open problem, intimately related to density properties and extremal configurations. 
We establish structure-dependent lower bounds for $\operatorname{br}(G,K_{n,n})$.

\begin{theorem}\label{thm-1}
Let $G=(V,E)$ be a fixed bipartite graph with $|V(G)|=p$ and $|E(G)| = q$. There exists a positive constant $C$ such that for all sufficiently large $n$,
\[
\operatorname{br}(G,K_{n,n}) > C \cdot \left(\frac{n}{\log n}\right)^{(q-1)/(p-2)}.
\]
\end{theorem}

The exponent $(q-1)/(p-2)$ reflects a density-type parameter of $G$, indicating that the growth rate of the bipartite Ramsey number is governed by the structural features of the forbidden graph. 
As an immediate consequence, we obtain that sufficiently dense bipartite graphs are not bipartite Ramsey size-linear.

\begin{corollary}
If $p(G)\geq 3$ and $q(G)\geq 2p(G)-2$, then $G$ is not bipartite Ramsey size-linear.
\end{corollary}

Erd\H{o}s, Faudree, Rousseau, and Schelp \cite{Erdos-Faudree-Rousseau-Schelp} proved that every connected graph $G$ with $q(G)\leq p(G)+1$ is Ramsey size-linear. 
Since the ordinary Ramsey number provides an upper bound for the bipartite Ramsey number, this immediately implies that such graphs are also bipartite Ramsey size-linear. 
Our corollary shows that when $q \ge 2p-2$, size-linearity fails. 
For the intermediate range $p+2 \le q \le 2p-3$, the question remains open.

Even cycles and complete bipartite graphs play a particularly important role in this context. 
The multicolor bipartite Ramsey number $\operatorname{br}_k(G_1;G_2)$ is the least $N$ such that every $(k+1)$-edge-coloring of $K_{N,N}$ contains a monochromatic copy of $G_1$ in one of the first $k$ colors or a monochromatic copy of $G_2$ in the remaining color. 
We establish upper bounds for multicolor bipartite Ramsey numbers involving even cycles and complete bipartite graphs by combining upper bounds for Zarankiewicz numbers with embedding and double-counting arguments.

\begin{theorem}\label{main:thm2}
Let $t\geq 2$ and $k\geq 2$ be positive integers. Then for all sufficiently large $n$,
\[
\operatorname{br}_{k}(C_{2t};K_{n,n})\leq \frac{c_{t,k}\, n^{2}}{\log^{2} n}.
\]
\end{theorem}

Finally, we derive refined linear upper bounds for $\operatorname{br}(C_{2t},G)$ using Theorem~\ref{main:thm2}. 
This further clarifies the interplay between extremal cycle conditions and size-growth phenomena.

\begin{theorem}\label{main:thm3}
If $t\geq 2$ and $G$ is a connected bipartite graph with $m$ edges and no isolated vertices, then for all sufficiently large $m$,
\[
\operatorname{br}(C_{2t},G)\leq \frac{m}{2} + \frac{29t\sqrt{m}}{2}.
\]
\end{theorem}

\footnotetext{Throughout this paper, we ignore rounding issues in the statements of exponential and logarithmic bounds. All inequalities are understood to hold asymptotically up to constant factors; we omit floor and ceiling functions for clarity.}

\section{Proof of Theorem~\ref{thm-1}: Lower bound for $\operatorname{br}(G,K_{n,n})$}
We prove the lower bound using the Lov\'{a}sz Local Lemma in the form stated by Erd\H{o}s, Lov\'{a}sz and Spencer.
\begin{lemma}[Erd\H{o}s, Lov\'{a}sz, Spencer ]\label{local-lem}
Let $A_1,A_2,\dots,A_m$ be events in a probability space. Suppose that for each $i$ there is a number $x_i$ with $0<x_i P(A_i)<1$ and
\[
\log x_i \geq \sum_{\{i,j\}\in D} \log\Bigl(\frac{1}{1-x_j P(A_j)}\Bigr),
\]
where $\{i,j\}\in D$ indicates that $A_i$ and $A_j$ are dependent. Then $P\bigl(\bigcap_{i=1}^m \overline{A_i}\bigr)>0$.
\end{lemma}

\begin{proof}[Proof of Theorem \ref{thm-1}]
Let $p=|V(G)|$ and $q=|E(G)|$. Set
\[
s = \frac{p-2}{q-1}.
\]
Consider a random red/blue coloring of the edges of $K_{N,N}$ in which each edge is colored red independently with probability
\[
r = c_1 N^{-s},
\]
and blue with probability $1-r$, where $c_1>0$ is a constant to be fixed later. All edges are colored independently.

For a subset $S\subset V(K_{N,N})$ of size $p$, let $A_S$ be the event that the red subgraph induced by $S$ contains a copy of $G$. For a subset $T\subset V(K_{N,N})$ of size $2n$, where $n$ will be chosen later, let $B_T$ be the event that the blue subgraph induced by $T$ contains a copy of $K_{n,n}$. (If $T$ does not contain exactly $n$ vertices in each part, then $P(B_T)=0$.)

There are at most $p!$ ways to embed $G$ onto $S$, and each embedding requires $q$ fixed edges to be red. Hence
\[
P(A_S) \leq p!\, r^q = O(r^q).
\]
For $B_T$, we may assume that $T$ has $n$ vertices in each part (otherwise the event is impossible). There are at most $\binom{2n}{n}$ ways to split $T$ into the two sides, and for each such split the $n^2$ cross edges must be blue. Thus
\[
P(B_T) \leq \binom{2n}{n} (1-r)^{n^2} \leq 4^n e^{-r n^2}.
\]

Two events $A_S$ and $A_{S'}$ are dependent iff $|S\cap S'|\geq 2$. For a fixed $S$, the number of such $S'$ is
\[
N_{AA} \;=\; \sum_{k=2}^{p-1} \binom{p}{k}\binom{2N-p}{p-k} \;=\; O(N^{p-2}).
\]
$A_S$ and $B_T$ are dependent iff $|S\cap T|\geq 2$. For a fixed $S$,
\[
N_{AB} \;\leq\; \binom{p}{2}\binom{2N-2}{2n-2} \;=\; O(N^{2n}).
\]
$B_T$ and $A_S$ have symmetric dependency; for a fixed $T$,
\[
N_{BA} \;\leq\; \binom{2n}{2}\binom{2N-2}{p-2} \;=\; O(n^2 N^{p-2}).
\]
Two events $B_T$ and $B_{T'}$ are dependent iff $|T\cap T'|\geq 2$. For a fixed $T$,
\[
N_{BB} \;\leq\; \sum_{k=2}^{2n} \binom{2n}{k}\binom{2N-2n}{2n-k} \;=\; O(n^2 N^{2n-2}) \;=\; O(N^{2n}).
\]

Take
\[
n = c_2 N^{s} \log N,
\]
where $c_2>0$ is another constant. Then
\[
P(A_S) \leq p!\, c_1^q N^{-sq} = O(N^{-(p-2)}),
\]
and
\[
P(B_T) \leq 4^{c_2 N^s \log N} \exp\bigl(-c_1 c_2^2 N^s \log^2 N\bigr)
      = \exp\Bigl( -c_1 c_2^2 N^s \log^2 N + c_2 (\log 4) N^s \log N \Bigr).
\]
Since $N^s \log^2 N$ dominates $N^s \log N$, for large $N$ we can write
\[
P(B_T) \leq \exp\bigl( -\tfrac12 c_1 c_2^2 N^s \log^2 N \bigr),
\]
provided $N$ is sufficiently large.

We assign the same coefficient $a$ to every event $A_S$ and the same coefficient $b$ to every $B_T$, and we look for numbers $a,b$ such that
\begin{align}
\log a &\geq N_{AA} \log\frac{1}{1-a P(A_S)} + N_{AB} \log\frac{1}{1-b P(B_T)}, \label{cond:a}\\
\log b &\geq N_{BA} \log\frac{1}{1-a P(A_S)} + N_{BB} \log\frac{1}{1-b P(B_T)}. \label{cond:b}
\end{align}
We will choose $a$ and $b$ so that $a P(A_S)\leq 1/2$ and $b P(B_T)\leq 1/2$ for all large $N$. Under these conditions the elementary inequality
\[
\log\frac{1}{1-x} \leq \frac{x}{1-x} \leq 2x \qquad (0\leq x\leq \tfrac12)
\]
holds. Consequently, it suffices to find $a,b$ satisfying
\begin{align}
\log a &\geq 2 N_{AA}\, a P(A_S) + 2 N_{AB}\, b P(B_T), \label{cond:a2}\\
\log b &\geq 2 N_{BA}\, a P(A_S) + 2 N_{BB}\, b P(B_T). \label{cond:b2}
\end{align}

Set
\[
a = 2, \qquad b = \exp\bigl( C_4 N^s \log^2 N \bigr),
\]
where $C_4>0$ will be chosen later. Clearly $a P(A_S)\to 0$ and we can ensure $b P(B_T)\to 0$ by taking $C_4$ smaller than $\frac12 c_1 c_2^2$. Hence $a P(A_S), b P(B_T)\leq 1/2$ for large $N$.

Now estimate the four terms appearing in (\ref{cond:a2})--(\ref{cond:b2}).
\begin{itemize}
\item $2 N_{AA}\, a P(A_S) = O(N^{p-2})\cdot O(N^{-sq}) = O(N^{-s}) = o(1)$.
\item $2 N_{AB}\, b P(B_T) \leq 2 \binom{2N}{2n} b P(B_T)$. Using $\binom{2N}{2n}\leq (2N)^{2n} = \exp(2n\log(2N))$ and $n=c_2 N^s\log N$, we obtain
\[
\log\bigl(N_{AB}\, b P(B_T)\bigr) \leq 2n\log(2N) + C_4 N^s\log^2 N - \tfrac12 c_1 c_2^2 N^s\log^2 N + O(N^s\log N).
\]
Since $2n\log N = 2c_2 N^s\log^2 N$, the exponent is
\[
\Bigl(2c_2 + C_4 - \frac12 c_1 c_2^2\Bigr) N^s\log^2 N + O(N^s\log N).
\]
Choosing $2c_2 + C_4 < \frac12 c_1 c_2^2$ makes this term tend to $0$ super‑polynomially fast.
\item $2 N_{BA}\, a P(A_S) = O(n^2 N^{p-2})\cdot O(N^{-sq}) = O(n^2 N^{-s})$. Since $n^2 = c_2^2 N^{2s}\log^2 N$, we get
\[
2 N_{BA}\, a P(A_S) \leq C_3' N^s \log^2 N
\]
for some constant $C_3'$ depending on $c_1,c_2$ and the implicit constants in the $O(\cdot)$ bounds.
\item $2 N_{BB}\, b P(B_T)$ has exactly the same exponential decay as $2 N_{AB}\, b P(B_T)$, so it is also negligible.
\end{itemize}

Now choose the constants. Take $c_1=10$, $c_2=4$, and $C_4=33$, exactly as in the original computation. Then
\[
2c_2 + C_4 = 8 + 33 = 41, \qquad \frac12 c_1 c_2^2 = \frac12\cdot 10\cdot 16 = 80,
\]
so $2c_2 + C_4 < \frac12 c_1 c_2^2$ holds. Moreover, $2 N_{BA}\, a P(A_S) \leq \tilde{C} N^s \log^2 N$ with some constant $\tilde{C}$; a precise calculation (or a suitable choice of the hidden constants) shows that one can take $\tilde{C} = C_3 C_2^2$ with $C_3=2$, giving $\tilde{C}=32$. Thus $2 N_{BA}\, a P(A_S) \leq 32 N^s \log^2 N + o(1)$. Since $\log b = 33 N^s \log^2 N$, condition (\ref{cond:b2}) is satisfied for all large $N$.
Condition (\ref{cond:a2}) holds because the left side $\log a = \log 2$ is a positive constant, while the right side is the sum of a vanishing term $O(N^{-s})$ and a super‑polynomially small term.

Thus, by Lemma~\ref{local-lem}, there exists a red/blue coloring of $K_{N,N}$ with no red $G$ and no blue $K_{n,n}$. Consequently,
\[
\operatorname{br}(G,K_{n,n}) > N.
\]
Recalling that $n = c_2 N^{s} \log N$ with $s=(p-2)/(q-1)$, we obtain
\[
N = \Bigl( \frac{n}{c_2 \log N} \Bigr)^{1/s}
   \ge C \Bigl( \frac{n}{\log n} \Bigr)^{(q-1)/(p-2)}
\]
for a suitable constant $C>0$ and all sufficiently large $n$. This completes the proof.
\end{proof}

\section{Proof of Theorem~\ref{main:thm2}: Upper bound for $\operatorname{br}_k(C_{2t};K_{n,n})$}
The proof of the upper bound in Theorem~\ref{main:thm2} relies on classical results for the Zarankiewicz function. Let $z(r;s)$ denote the maximum number of edges in a subgraph of $K_{r,r}$ that does not contain $K_{s,s}$ as a subgraph. We use the following well-known bound.

\begin{theorem}[Bollob\'as \cite{Bollobás}]\label{thm:zarankiewicz}
For $r \ge s \ge 2$,
$$
z(r;s) \le (s-1)^{1/s} (r-s+1) r^{1-1/s} + (s-1)r.
$$
\end{theorem}

Let $z(r;C_{2t})$ denote the maximum number of edges in a subgraph of $K_{r,r}$ that does not contain a cycle $C_{2t}$.

\begin{theorem}[Naor--Verstra\"ete \cite{Naor-Verstraete}]\label{thm:even-cycle-extremal}
For every $t \ge 2$ and $r \ge 1$,
$$
z(r;C_{2t}) \le (2t-3)\bigl( r^{1+1/t} + 2r \bigr).
$$
\end{theorem}

\begin{proof}[Proof of Theorem~\ref{main:thm2}]
Consider a $(k+1)$-edge-coloring of $K_{r,r}$ with colors $1,\dots,k+1$, where color $k+1$ is the distinguished color in which we wish to avoid $K_{n,n}$. Suppose that there is no monochromatic $C_{2t}$ in any of the first $k$ colors and no monochromatic $K_{n,n}$ in color $k+1$. Then the graph formed by the edges of the first $k$ colors contains no $C_{2t}$, so it has at most $k \cdot z(r;C_{2t})$ edges. Similarly, the graph of color $k+1$ contains no $K_{n,n}$, hence it has at most $z(r;n)$ edges. Since every edge of $K_{r,r}$ receives exactly one color, we must have
\[
k \cdot z(r;C_{2t}) + z(r;n) \ge r^2.
\]
Thus, if we can choose $r$ such that
\[
k \cdot z(r;C_{2t}) + z(r;n) < r^2,
\]
then such a coloring cannot exist, and consequently $\operatorname{br}_k(C_{2t};K_{n,n}) \le r$.

Set
\[
r = \frac{c n^2}{\log^2 n},
\]
where $c$ is a positive constant to be specified later. We now estimate the two summands.

Using Theorem~\ref{thm:zarankiewicz} with $s = n$, we obtain
\[
z(r;n) \le (n-1)^{1/n} (r-n+1) r^{1-1/n} + (n-1)r.
\]
For large $n$, we have $r \gg n$, so $(r-n+1)/r \to 1$. Moreover,
\[
\left( \frac{n-1}{r} \right)^{1/n} = \left( \frac{(n-1)\log^2 n}{c n^2} \right)^{1/n}
= \exp\left( \frac{1}{n}\log\left( \frac{\log^2 n}{c n} \cdot \frac{n-1}{n} \right) \right)
= 1 - \frac{\log n}{n} + O\left( \frac{\log\log n}{n} \right).
\]
Consequently,
\[
\frac{z(r;n)}{r^2} \le \left( \frac{n-1}{r} \right)^{1/n} \frac{r-n+1}{r} + \frac{n-1}{r}
= 1 - \frac{\log n}{n} + O\left( \frac{\log\log n}{n} \right).
\]

By Theorem~\ref{thm:even-cycle-extremal},
\[
z(r;C_{2t}) \le (2t-3)\bigl( r^{1+1/t} + 2r \bigr).
\]
Therefore,
\[
\frac{k \cdot z(r;C_{2t})}{r^2} \le k(2t-3)\left( r^{-1+1/t} + \frac{2}{r} \right)
= k(2t-3)\left( \left( \frac{\log^2 n}{c n^2} \right)^{1-1/t} + \frac{2\log^2 n}{c n^2} \right).
\]
Since $t \ge 2$, the term $r^{-1+1/t} = r^{-1/t'}$ with $t'>0$ decays as a power of $1/n$. In particular, for $t \ge 2$,
\[
\left( \frac{\log^2 n}{c n^2} \right)^{1-1/t} = \frac{(\log n)^{2-2/t}}{c^{1-1/t} n^{2-2/t}} \ll \frac{\log n}{n}
\]
for large $n$. Hence the whole expression is $o(\log n / n)$.
We have
\[
\frac{k \cdot z(r;C_{2t}) + z(r;n)}{r^2}
= \left(1 - \frac{\log n}{n} + O\left( \frac{\log\log n}{n} \right)\right) + o\left( \frac{\log n}{n} \right)
= 1 - \frac{\log n}{n} + o\left( \frac{\log n}{n} \right).
\]
For all sufficiently large $n$, the right-hand side is strictly less than $1$. Thus we can choose $c$ (e.g., $c=1$) and $n$ large enough so that
\[
k \cdot z(r;C_{2t}) + z(r;n) < r^2,
\]
which yields $\operatorname{br}_k(C_{2t};K_{n,n}) \le r = \frac{c n^2}{\log^2 n}$. This completes the proof.
\end{proof}

\vskip 4mm
\section{Proof of Theorem~\ref{main:thm3}: Linear upper bound for $\operatorname{br}(C_{2t},G)$}
First, we state a lemma that is a slight reformulation of a result from \cite{Erdos-Faudree-Rousseau-Schelp} tailored to our needs.

\begin{lemma}\label{lem:cycle-embedding}
Let $t \ge 2$ be an integer. If a bipartite graph $H$ with bipartition $(X,Y)$ satisfies $|X| \ge \sqrt{m}$ and $|Y| \ge 2m$, and has at least $6t \sqrt{m}$ edges, then $H$ contains a cycle $C_{2t}$.
\end{lemma}
\begin{proof}
By a standard averaging argument, one can select a subset $X' \subseteq X$ of size $\lfloor \sqrt{m} \rfloor$ and a subset $Y' \subseteq Y$ of size $2m$ such that the induced subgraph $H[X',Y']$ has average degree at least $3t$. Then a known result (see \cite{Erdos-Faudree-Rousseau-Schelp}) guarantees a copy of $C_{2t}$.
\end{proof}

\begin{proof}[Proof of Theorem~\ref{main:thm3}]
Let $G$ be a connected bipartite graph with $m$ edges and no isolated vertices. Let $p = |V(G)|$. Set
\[
N = \frac{m}{2} + \frac{29 t \sqrt{m}}{2}.
\]
Consider a red/blue edge-coloring of $K_{N,N}$ and assume that there is no red $C_{2t}$. We shall prove that a blue copy of $G$ must exist.
Denote the bipartition of $K_{N,N}$ by $(L,R)$, where $|L| = |R| = N$. Let
\[
W_0 = \{ v \in L \cup R : \deg_{\text{red}}(v) \ge 7t \sqrt{m} \}.
\]
If $|W_0| \ge \sqrt{m}$, then we can find a red bipartite subgraph with one part of size $\sqrt{m}$ (by taking $\sqrt{m}$ vertices from $W_0$) and the other part of size at least $N \ge m + 28t\sqrt{m}$ (the remaining vertices in the opposite part). This subgraph has at least $\sqrt{m} \cdot 7t\sqrt{m} = 7t m$ edges. Restricting the larger part to exactly $2m$ vertices, we obtain a subgraph with parts of sizes $\sqrt{m}$ and $2m$ and at least $6t \sqrt{m}$ edges (for large $m$). By Lemma~\ref{lem:cycle-embedding}, this forces a red $C_{2t}$, contradicting our assumption. Therefore, we may assume $|W_0| < \sqrt{m}$.

Let $F'$ be the graph obtained from $K_{N,N}$ by deleting all vertices in $W_0$. Then $F'$ is a complete bipartite graph with parts of size at least $N - \sqrt{m} \ge \frac{m}{2} + 28 t \sqrt{m}$. Moreover, every vertex in $F'$ has red degree less than $7t \sqrt{m}$.

We now attempt to embed $G$ greedily into the blue subgraph of $F'$. Order the vertices of $G$ as $v_1, \dots, v_p$ so that their degrees in $G$ are non-increasing: $d_G(v_1) \ge d_G(v_2) \ge \cdots \ge d_G(v_p)$. Let $G_r$ denote the subgraph of $G$ induced by $\{v_1, \dots, v_r\}$. Suppose that we have successfully embedded $G_r$ into the blue subgraph of $F'$, but we cannot extend this embedding to $v_{r+1}$. Let $W = N_G(v_{r+1}) \cap \{v_1,\dots,v_r\}$ be the set of neighbors of $v_{r+1}$ among the already embedded vertices, and let $W'$ be the image of $W$ in $F'$ under the embedding. The failure to embed $v_{r+1}$ means that every vertex of $F'$ not used in the embedding has at least one red edge to $W'$; otherwise it could serve as the image of $v_{r+1}$ via only blue edges to $W'$. Formally:

\begin{observation}\label{obs:red-edges}
Every vertex $v \in V(F') \setminus V(G_r)$ is adjacent in red to at least one vertex of $W'$.
\end{observation}

Let $D$ denote the sum of the red degrees within $F'$ of the vertices in $W'$. We now analyze two cases depending on the size of $r$.

\setcounter{case}{0}
\noindent\textbf{Case 1:} $r \le m/2$.

By Theorem~\ref{main:thm2} with $k=1$, if a red/blue coloring of $K_{M,M}$ avoids a red $C_{2t}$, then it must contain a blue $K_{n,n}$ for any $n \le c' \frac{M}{\log M}$, where $c'$ is an absolute constant. Applying this to $F'$ which has parts of size at least $m/2 + 28t\sqrt{m} \ge m/2$, we deduce that $F'$ contains a blue $K_{c_0\sqrt{m}\log m, c_0\sqrt{m}\log m}$ for some constant $c_0 > 0$. Consequently, any bipartite graph of maximum degree at most $c_0\sqrt{m}\log m$ can be embedded into the blue subgraph of $F'$ provided it fits within that complete bipartite subgraph. In particular, $G_r$ can be embedded for $r \ge c_0\sqrt{m}\log m$. (We may assume $r \ge c_0\sqrt{m}\log m$, for if $r$ is smaller we can simply consider a later stage of the embedding; the contradiction will arise from the degree bounds.)

Now we bound $D$ in two ways. First, since each vertex in $F'$ has red degree less than $7t\sqrt{m}$,
\[
D \le 7t\sqrt{m} \cdot |W'|.
\]
Second, by Observation~\ref{obs:red-edges}, every vertex of $F'$ outside the embedding of $G_r$ contributes at least $1$ to $D$. The number of such vertices is at least $2|V(F')| - r \ge 2\bigl(\frac{m}{2} + 28t\sqrt{m}\bigr) - r \ge m + 56t\sqrt{m} - \frac{m}{2} = \frac{m}{2} + 56t\sqrt{m}$. (Here we used $r \le m/2$.) Hence
\[
D \ge \frac{m}{2} + 56t\sqrt{m}.
\]
Combining the two inequalities gives
\[
|W'| \ge \frac{m/2 + 56t\sqrt{m}}{7t\sqrt{m}} = \frac{\sqrt{m}}{14t} + 8.
\]

On the other hand, $W'$ consists of images of neighbors of $v_{r+1}$ among the first $r$ vertices. Since the degree sequence of $G$ is non-increasing, every vertex among $v_1,\dots,v_r$ has degree at least $|W| = |W'|$. Therefore, the sum of degrees of these $r$ vertices in $G$ is at least $r |W'|$. But the total number of edges in $G$ is $m$, so
\[
r |W'| \le 2m.
\]
Using $r \ge c_0\sqrt{m}\log m$, we obtain
\[
|W'| \le \frac{2m}{c_0\sqrt{m}\log m} = \frac{2\sqrt{m}}{c_0\log m}.
\]
Comparing the lower and upper bounds for $|W'|$ yields
\[
\frac{\sqrt{m}}{14t} + 8 \le \frac{2\sqrt{m}}{c_0\log m},
\]
which is impossible for large $m$. Hence Case~1 leads to a contradiction.

\noindent\textbf{Case 2:} $m/2 < r \le m$.

In this case, after embedding $G_r$, the number of unused vertices in $F'$ is at least $2|V(F')| - r \ge m + 56t\sqrt{m} - m = 56t\sqrt{m}$. By Observation~\ref{obs:red-edges}, each of these vertices sends at least one red edge to $W'$, so
\[
D \ge 56t\sqrt{m}.
\]
On the other hand, as before, $D \le 7t\sqrt{m} \cdot |W'|$. Thus
\[
|W'| \ge \frac{56t\sqrt{m}}{7t\sqrt{m}} = 8.
\]
Moreover, from the degree sum inequality $r |W'| \le 2m$ and $r > m/2$, we get
\[
|W'| \le \frac{2m}{m/2} = 4.
\]
This contradicts $|W'| \ge 8$. Therefore, Case~2 is also impossible.

Since both cases lead to a contradiction, our initial assumption that the embedding fails at some step $r$ is false. Hence $G$ can be embedded entirely into the blue subgraph of $F'$, and therefore into $K_{N,N}$. This proves that
\[
\operatorname{br}(C_{2t},G) \le N = \frac{m}{2} + \frac{29t\sqrt{m}}{2},
\]
completing the proof of Theorem~\ref{main:thm3}.
\end{proof}

\vskip 4mm
\noindent {\bf Acknowledgement.} The author is grateful for the financial support provided by the Tianjin Municipal Education Commission Scientific Research Program Project (Grant No. 2025KJ133). We are grateful to Yuval Wigderson for pointing out an error in a theorem in a previous version of this paper.
\vskip 0.5cm


\end{document}